\documentclass[a4paper,11pt]{article}
\usepackage[a4paper, %
left=2.9cm, %
right=2.9cm, %
top=2.9cm, %
bottom=3cm, %
heightrounded]{geometry}
\usepackage{amsmath}
\usepackage{amsthm}
\usepackage{amsfonts}
\usepackage{amssymb}
\usepackage{stmaryrd}
\usepackage{mathrsfs}
\usepackage{enumerate}
\usepackage{sectsty}

\newcommand{\N}{\mathbb{N}}
\newcommand{\Z}{\mathbb{Z}}

\newcommand{\R}{\mathbb{R}}
\newcommand{\C}{\mathbb{C}}
\newcommand{\T}{\mathbb{T}}

\sectionfont{\centering}

\author{Vincent DEVINCK}

\newtheorem{Theo}{Theorem}[section]
\newtheorem{Lem}[Theo]{Lemma}
\newtheorem{Prop}[Theo]{Proposition}
\newtheorem{Cor}[Theo]{Corollary}
\newtheorem{Fac}[Theo]{Fact}

\theoremstyle{definition}
\newtheorem{Def}[Theo]{Definition}
\newtheorem{Ex}[Theo]{Example}

\newtheorem{Que}[Theo]{Question}
\newtheorem{Rem}[Theo]{Remark}

\numberwithin{equation}{section}

\date{}
\title{\textbf{Universal Jamison spaces and Jamison sequences for $C_0$-semigroups}}
\begin{document}
\maketitle

\begin{abstract}
An increasing sequence of positive integers $(n_k)_{k\ge 0}$ is said to be a Jamison sequence if the following property holds true: for every separable complex Banach space $X$ and every $T\in \mathcal{B}(X)$ which is partially power-bounded with respect to $(n_k)_{k\ge 0}$, the set $\sigma_p(T)\cap \T$ is at most countable. We prove that a separable infinite-dimensional complex Banach space $X$ which admits an unconditional Schauder decomposition is such that for any sequence $(n_k)_{k\ge 0}$ which is not a Jamison sequence, there exists $T\in \mathcal{B}(X)$ which is partially power-bounded with respect to this sequence and such that the set $\sigma_p(T)\cap \T$ is uncountable. We also investigate the notion of Jamison sequences for $C_0$-semigroups and we give an arithmetic characterization of these sequences.
\end{abstract}

\section{Introduction}

Let $X$ be a separable infinite-dimensional complex Banach space and $T\in \mathcal{B}(X)$ a bounded linear operator on $X$. In the whole paper, $\T=\big\{\lambda\in\C\,;\, \vert\lambda\vert=1\big\}$ stands for the unit circle of the complex plane and $\sigma_p(T)=\big\{\lambda\in\C\,;\, \mathrm{Ker}(T-\lambda)\ne \{0\}\big\}$ is the point spectrum of $T$. The set $\sigma_p(T)\cap \T$ will be called the \textit{unimodular point spectrum} of $T$. It is now well-known that the behaviour of the sequence $(\vert\vert T^n\vert\vert)_{n\ge 0}$ of the norms of the iterates of $T$ is closely related to the size of the unimodular point spectrum $\sigma_p(T)\cap \T$: Jamison proved in \cite{Jamison} that if $T$ is power-bounded, that is to say if $\sup_{n\ge 0}\vert\vert T^n\vert\vert<+\infty$, then the unimodular point spectrum $\sigma_p(T)\cap \T$ of $T$ is at most countable. The influence of (partial) power-boundedness on the size of $\sigma_p(T)\cap\T$ has been studied by Ransford \cite{Ransford}, Ransford and Roginskaya \cite{RansfordRoginskaya}, Badea and Grivaux (\cite{BadeaGrivaux1} and \cite{BadeaGrivaux2}) and more recently by Grivaux and Eisner \cite{EisnerGrivaux}.

If $(n_k)_{k\ge 0}$ is an increasing sequence of positive integers, we say that the operator $T\in \mathcal{B}(X)$ is \textit{partially power-bounded with respect to the sequence} $(n_k)_{k\ge 0}$ if $\sup_{k\ge 0}\vert\vert T^{n_k}\vert\vert<+\infty$. We say that the sequence $(n_k)_{k\ge 0}$ is a \textit{Jamison sequence} if for every separable complex Banach space $X$ and every bounded linear operator $T$ on $X$ which is partially power-bounded with respect to the sequence $(n_k)_{k\ge 0}$, the set $\sigma_p(T)\cap \T$ is at most countable. The notion of Jamison sequence has been intensely studied by the previously mentioned authors, and C. Badea and S. Grivaux found an arithmetic characterization of Jamison sequences in \cite{BadeaGrivaux2}. Under the assumption that $n_0=1$, one can define a distance $d_{(n_k)}$ on $\T$ by setting
$$
\forall (\lambda,\mu)\in\T^2,\qquad d_{(n_k)}(\lambda,\mu)=\sup_{k\ge 0}\vert \lambda^{n_k}-\mu^{n_k}\vert.
$$
The characterization of \cite{BadeaGrivaux2} runs as follows.
\begin{Theo}\emph{(\cite[Theorem 2.1]{BadeaGrivaux2})}\label{jamisonth} Let $(n_k)_{k\ge 0}$ be an increasing sequence of positive integers such that $n_0=1$. The following assertions are equivalent$:$\\
\noindent $(1)$ $(n_k)_{k\ge 0}$ is a Jamison sequence$;$\\
\noindent $(2)$ there exists $\varepsilon>0$ such that any two distinct points of $\T$ are $\varepsilon$-separated for the distance $d_{(n_k)}:$
$$
\forall (\lambda,\mu)\in \T^2,\, \lambda\ne \mu\, \Longrightarrow\, \sup_{k\ge 0}\vert \lambda^{n_k}-\mu^{n_k}\vert\ge \varepsilon.
$$
\end{Theo}

The hard part of the proof of Theorem \ref{jamisonth} is the following: when condition $(2)$ is not	satisfied, C. Badea and S. Grivaux had to construct a separable Banach space $X$ and a bounded linear operator $T$ on $X$ which is partially power-bounded with respect to $(n_k)_{k\ge 0}$ and such that the set $\sigma_p(T)\cap \T$ is uncountable. In this paper, we are interested in this problem of construction. We can state the problem as follows.

\begin{Que}\label{question}
If $(n_k)_{k\ge 0}$ is not a Jamison sequence, on which separable complex Banach spaces $X$ can we construct a partially power-bounded operator $T$ with respect to the sequence $(n_k)_{k\ge 0}$ with an uncountable unimodular point spectrum?
\end{Que}

A separable complex Banach space for which the answer to the above question is affirmative will be called a \textit{universal Jamison space} (Definition \ref{defjamisonuniversel}). Eisner and Grivaux proved in \cite{EisnerGrivaux} that a separable infinite-dimensional complex Hilbert space is a universal Jamison space. Recall that if $T$ is a bounded linear operator on a separable complex Banach space $X$, one says that $T$ has a \textit{perfectly spanning set of eigenvectors associated to unimodular eigenvalues} (see \cite{BG1}) if there exists a continuous probability measure $\sigma$ on the unit circle $\T$ such that for any Borel subset $B$ of $\T$ with $\sigma(B)=1$, we have
$$
\overline{\textrm{span}}\big[\mathrm{Ker}(T-\lambda)\,;\, \lambda\in B\big]=X.
$$
Grivaux proved in \cite[Theorem $4.1$]{Grivaux} that $T$ has a perfectly spanning set of eigenvectors associated to unimodular eigenvalues if and only if for any countable subset $D$ of $\T$, 
$$
\overline{\textrm{span}}\big[\mathrm{Ker}(T-\lambda)\,;\, \lambda\in \T\setminus D\big]=X.
$$

\begin{Theo}\emph{(\cite[\textrm{Theorem}\ 2.1]{EisnerGrivaux})}\label{hilbertjamisonuniversel}
Let $(n_k)_{k\ge 0}$ be an increasing sequence of positive integers $($with $n_0=1)$ such that for any $\varepsilon>0$ there exists $\lambda\in \T\setminus \{1\}$ such that
$$
\sup_{k\ge 0}\vert \lambda^{n_k}-1\vert \le \varepsilon.
$$
Let $\delta$ be any positive real number. There exists a bounded linear operator $T$ on the complex Hilbert space $\ell_2(\N)$ such that $T$ has a perfectly spanning set of eigenvectors associated to unimodular eigenvalues and
$$
\sup_{k\ge 0}\vert\vert T^{n_k}\vert\vert\le 1+\delta.
$$
In particular the unimodular point spectrum of $T$ is uncountable.
\end{Theo}

In section 2 of this paper, we generalize Theorem \ref{hilbertjamisonuniversel} by providing a large class of separable complex Banach spaces (the class of Banach spaces which admit unconditional Schauder decompositions) which are universal Jamison spaces (Theorem \ref{theoremejamisonuniversel}). We also give several concrete examples of Banach spaces which are universal Jamison spaces and we prove that, in opposition with the spaces which admit an unconditional Schauder decomposition, the hereditarily indecomposable spaces are never universal Jamison spaces.

 In section 3, we investigate the subject of \textit{Jamison sequences for }$C_0$-\textit{semigroups} (Definition \ref{jamisonsemigroupes}) which are the analog of Jamison sequences in the context of semigroups. We give an arithmetic characterization of these sequences (Theorem \ref{caracterisationjamisonsemigroupes}) by using the characterization of Jamison sequences (Theorem \ref{jamisonth}). We also consider \textit{universal Jamison spaces for} $C_0$-\textit{semigroups} (Definition \ref{defuniversaljamisonsg}) and prove that every separable complex Banach space which admits an unconditional Schauder decomposition is a universal Jamison space for $C_0$-semigroups (Theorem \ref{espacejamisonuniverselsemigroupes}).
 
At the end of section 3, we study the Hausdorff dimension of the unimodular point spectrum in the context of semigroups. We prove that if  $(t_k)_{k\ge 0}$ is an increasing sequence of positive real numbers such that $t_0=1$ and $\frac{t_{k+1}}{t_k} \longrightarrow +\infty$, there exists a separable complex Banach space $X$ and a $C_0$-semigroup $(T_t)_{t\ge 0}$ of bounded linear operators on $X$ (with infinitesimal generator $A$) with $\sup_{k\ge 0}\vert\vert T_{t_k}\vert\vert<+\infty$ and such that the set $\sigma_p(A)\cap i\R$ has Hausdorff dimension equal to $1$ (Theorem \ref{hausdorff1}). 

\section{Universal Jamison spaces}
In this section, we investigate the notion of universal Jamison spaces which is the kind of Banach spaces which have been mentioned in Question \ref{question}.

\begin{Def}\label{defjamisonuniversel}
Let $X$ be a separable infinite-dimensional complex Banach space. We say that $X$ is a \textit{universal Jamison space} if the following property holds true: for any increasing sequence of positive integers $(n_k)_{k\ge 0}$ which is not a Jamison sequence, there exists $T\in \mathcal{B}(X)$ which is partially power-bounded with respect to the sequence $(n_k)_{k\ge 0}$ and which has an uncountable unimodular point spectrum.
\end{Def}

\begin{Ex}
According to Theorem \ref{hilbertjamisonuniversel}, the space $\ell_2(\N)$ is a universal Jamison space. More generaly, if we replace $2$ by $p\in [1,+\infty[$ in the proof of Theorem \ref{hilbertjamisonuniversel}, we easily see that the space $\ell_p(\N)$ is also a universal Jamison space. 
\end{Ex}

Our aim is to generalize Theorem \ref{hilbertjamisonuniversel} to a broader class of Banach spaces $X$, namely to the class of separable complex Banach spaces which admit an unconditional Schauder decomposition.

\subsection{Unconditional Schauder decompositions}
Let us now recall briefly a few known facts about unconditional Schauder decompositions.

\begin{Def}\label{decschauderinconditionnelle}
Let $X$ be a separable infinite-dimensional complex space. We say that $X$ admits an unconditional Schauder decomposition if there exists a sequence $(X_\ell)_{\ell\ge 1}$ of closed subspaces of $X$ (different from $\{0\}$) such that any vector $x$ of $X$ can be written in a unique way as an unconditionally convergent series $\sum_{\ell\ge 1}x_\ell$, where $x_\ell$ belongs to $X_\ell$ for all positive integers $\ell$. 
\end{Def}

There are many examples of spaces which admit an unconditional Schauder decomposition. For instance, it is clear that if $X$ has an unconditional Schauder basis, then $X$ admits an unconditional Schauder decomposition. Recall that the space $C([0,1])$ of continuous functions on $[0,1]$ is \textit{universal} in the sense that it contains an isometric copy of any separable Banach space. We denote by $c_0(\mathbb{N})$ the space of all complex sequences which converge to zero. 

\begin{Ex}\label{fonctionscontinues}
A separable complex Banach space $X$ admits an unconditional Schauder decomposition whenever it contains a complemented copy of a Banach space which admits an unconditional Schauder decomposition. In particular, if $X$ contains a copy of $c_0(\mathbb{N})$ then it admits an unconditional Schauder decomposition. For instance, the space $C([0,1])$ admits an unconditional Schauder decomposition.
\end{Ex}

\begin{proof}
In the first case, the fact that $X$ admits an unconditional Schauder decomposition follows directly from Definition \ref{decschauderinconditionnelle}. If $X$ contains a copy of $c_0(\mathbb{N})$, then a result of Sobczyk says that this copy is complemented in $X$. The last assertion comes from the fact that $C([0,1])$ is a universal space and so contains a copy of $c_0(\mathbb{N})$.
\end{proof}

\begin{Rem}
If $(X_\ell)_{\ell\ge 1}$ is an unconditional Schauder decomposition of $X$ and $(I_k)_{k\ge 1}$ is any partition of $\N$ into finite or infinite subsets, let $Y_k$ denote the closed linear span of the spaces $X_\ell$, where $\ell\in I_k$. Then $(Y_k)_{k\ge 1}$ is also an unconditional Schauder decomposition of $X$. Hence, we will always suppose in the sequel that $(X_\ell)_{\ell \ge 1}$ is an unconditional Schauder decomposition of $X$ with all the subspaces $X_\ell$ infinite-dimensional.
\end{Rem}

This assumption will allow us to define a weighted backward shift on a space which admits an unconditional Schauder decomposition. To do this, recall the notion of \textit{biorthogonal system}. We denote by $\delta_{i,j}$ the Kronecker symbol.

\begin{Def}(\cite[Definition 1.f.1.]{LindenstraussTzafriri})
Let $Z$ be a separable infinite-dimensional complex Banach space. For any positive integer $i$, let $z_i$ and $z_i^*$ be elements of $Z$ and $Z^*$ respectively. The sequence $\big((z_i)_{i\ge 1},(z_i^{*})_{i\ge 1}\big)$ is called a biorthogonal system in $Z$ if $\langle z_i^{*},z_j\rangle=\delta_{i,j}$ for any positive integers $i, j$.
\end{Def}

The following result will be fundamental in the sequel.

\begin{Theo}\emph{(\cite[Theorem 1.f.4]{LindenstraussTzafriri})}\label{theoremesystemebiorthogonal}
Let $Z$ be a separable infinite-dimensional complex Banach space. there exists a biorthogonal sequence $\big((z_i)_{i\ge 1},(z_i^{*})_{i\ge 1}\big)$ in $Z$ $($where $z_i\in Z$ and $z_i^{*}\in Z^{*})$ such that$:$\\
\noindent $(1)$ $\sup_{i\ge 1}\vert\vert z_i\vert\vert\,\vert\vert z_i^{*}\vert\vert<+\infty;$\\
\noindent $(2)$ the linear span of the vectors $z_i$, $i\ge 1$, is a dense subspace of $Z;$\\
\noindent $(3)$ for every vector $z$ of $Z$ such that $\langle z_i^{*},z\rangle=0$ for any positive integer $i$, $z=0$.
\end{Theo}

For more informations on biorthogonal systems, we refer the reader to the book \cite{LindenstraussTzafriri}.

\subsection{The result}
We are now ready to prove our result about Jamison universal spaces. The proof of this result is very close to that of \cite[Theorem 2.1]{EisnerGrivaux}.

\begin{Theo}\label{theoremejamisonuniversel}
Let $X$ be a separable complex Banach space which admits an unconditional Schauder decomposition. Then $X$ is a universal Jamison space.
\end{Theo}
\begin{proof}
Let us fix an increasing sequence of positive integers $(n_k)_{k\ge 0}$ (with $n_0=1$) which is not a Jamison sequence. Our task is to construct a bounded linear operator $T$ on the space $X$ which is partially power-bounded with respect to the sequence $(n_k)_{k\ge 0}$ and such that the set $\sigma_p(T)\cap \T$ is uncountable. By definition of $X$, there exists a sequence $(X_\ell)_{\ell \ge 1}$ of infinite-dimensional closed subspaces of $X$ such that any vector $x$ of $X$ can be written in a unique way as an unconditionally convergent series $\sum_{\ell\ge 1}x_\ell$, where $x_\ell$ belongs to $X_\ell$ for all positive integers $\ell$. In particular, $X=\bigoplus_{\ell\ge 1}X_\ell$. Let us consider a biorthogonal system in each space $X_\ell$: according to Theorem \ref{theoremesystemebiorthogonal}, there exists a biorthogonal sequence $\big((x_{i,\ell})_{i\ge 1},(x_{i,\ell}^{*})_{i\ge 1}\big)$ in $X_\ell$ such that
\begin{equation}\label{conditionsysteme}
\vert\vert x_{i,\ell}\vert\vert=1\ \ \ \mathrm{for\ any\ }i\ge 1\ \ \ \mathrm{and}\ \ \ M_\ell:=\sup_{i\ge 1}\vert\vert x_{i,\ell}^{*}\vert\vert <+\infty.
\end{equation}
Our first task will be to define the operator $T$ and to prove that it is bounded on $X$. We will then study the eigenvectors of $T$ associated to the unimodular eigenvalues and give estimates of the norms $\vert\vert T^{n_p}\vert\vert$.

\noindent $\rhd$ \textbf{Definition of $T$.}\\
\noindent As in the proof of \cite[Theorem 2.1]{EisnerGrivaux}, we define $T$ as the sum of a diagonal operator and a weighted backward shift. The construction depends on two sequences $(\lambda_n)_{n\ge 1}$ and $(w_n)_{n\ge 1}$ which will be suitably chosen in the proof: $(\lambda_{n})_{n\ge 1}$ is a sequence of unimodular complex numbers which are all distinct and $(w_n)_{n\ge 1}$ is a sequence of positive weights. Let us first define the operator $D:X\longrightarrow X$ associated to the sequence $(\lambda_n)_{n\ge 1}$ by setting
\begin{equation*}
D\Bigg(\sum_{\ell \ge 1}x_\ell\Bigg)=\sum_{\ell\ge 1}\lambda_\ell x_\ell
\end{equation*}
for any vector $x=\sum_{\ell\ge 1}x_\ell$ of $X$. Since the decomposition $X=\bigoplus_{\ell \ge 1}X_\ell$ is unconditional, $D$ defines a bounded linear operator on $X$ which is partially power-bounded with respect to the sequence $(n_k)_{k\ge 0}$. Let us then define a weighted backward shift on $X$ by using the biorthogonal systems $\big((x_{i,\ell})_{i\ge 1},(x_{i,\ell}^{*})_{i\ge 1}\big)$. Since 
\begin{equation*}
X=X_\ell\bigoplus \overline{\mathrm{span}}\Bigg(\bigcup_{p\ne \ell}X_p\Bigg),
\end{equation*}
we first observe that we can extend the functionals $x_{i,\ell}^*$ to $X$ by setting $x_{i,\ell}^*=0$ on $\overline{\mathrm{span}}\big(\bigcup_{p\ne \ell}X_p\big)$. If we denote by $C>0$ the unconditional constant of the decomposition $X=\bigoplus_{\ell\ge 1}X_\ell$ then $\Vert x_{i,\ell}^*\Vert\le CM_\ell$ for any positive integers $i$ 	and $\ell$. Since $\langle x_{i,\ell}^*, x_{i,\ell}\rangle=1$ and $\Vert x_{i,\ell}\Vert=1$, we have $CM_\ell\ge 1$ for any positive integer $\ell$. Let $((w_{i,\ell})_{i\ge 1})_{\ell\ge 1}$ be a double sequence of positive real numbers which will be defined further on in the proof. Let finally $j: \N\setminus\{1\}\longrightarrow \N$ be a function which satisfies the two following conditions:
\begin{enumerate}
\item[$(i)$] for any $n\ge 2$, $j(n)<n$;
\item[$(ii)$] for any $k\ge 1$, the set $\{n\ge 2\,;\, j(n)=k\}$ is infinite.
\end{enumerate} 
Our weighted backward shift will be defined as
\begin{equation}\label{backwardshift}
Bx:=\sum_{\ell\ge 2}\sum_{i\ge 1}\langle x_{i,\ell}^{*},x\rangle \alpha_{i,\ell-1}x_{i,\ell-1}\qquad (x\in X)
\end{equation} 
where
\begin{equation}\label{poids1}
\alpha_{i,1}:=w_{i,1}\vert \lambda_2-\lambda_{j(2)}\vert\ \ \ \ \ \mathrm{for\ all\ }i\ge 1
\end{equation}
and
\begin{equation}\label{poids2}
\alpha_{i,\ell}:=w_{i,\ell}\frac{\vert \lambda_{\ell+1}-\lambda_{j(\ell+1)}\vert}{\vert \lambda_{\ell}-\lambda_{j(\ell)}\vert}\ \ \ \ \ \mathrm{for\ all\ }\ i\ge 1\ \mathrm{and\ }  \ell\ge 2,
\end{equation}
where $w_{i,\ell}>0$. We now choose the coefficients $\alpha_{i,\ell}$ in such a way that \eqref{backwardshift} defines a bounded linear operator on $X$. For any positive integers $i$ and $\ell$, let us define
\begin{equation}\label{poids}
w_{i,\ell}:=\frac{w_\ell}{2^i CM_{\ell+1}},
\end{equation}
where $w_\ell$ is an arbitrary positive real number. Under this condition, the series
\begin{equation*}
\mathcal{N}:=\sum_{\ell\ge 2}\sum_{i\ge 1}\alpha_{i,\ell-1}\,\vert\vert x_{i,\ell-1}\vert\vert\,\vert\vert x_{i,\ell}^{*}\vert\vert
\end{equation*}
is convergent. Indeed, conditions \eqref{conditionsysteme} and \eqref{poids} and Definitions \eqref{poids1} and \eqref{poids2} yields that
\begin{align*}
\mathcal{N}&=\sum_{\ell\ge 1} \sum_{i\ge 1}\alpha_{i,\ell}\,\vert\vert x_{i,\ell+1}^{*}\vert\vert\\
&\le w_1\vert \lambda_2-\lambda_{j(2)}\vert +\sum_{\ell=2}^{+\infty}w_\ell\frac{\vert \lambda_{\ell+1}-\lambda_{j(\ell+1)} \vert}{\vert \lambda_\ell-\lambda_{j(\ell)} \vert}\cdot
\end{align*}
For arbitrary weights $w_\ell>0$, we can take $\lambda_\ell$ sufficiently close to $\lambda_{j(\ell)}$ for any positive integer $\ell$ so that the series 
\begin{equation*}
\sum_{\ell\ge 2}w_\ell\frac{\vert \lambda_{\ell+1}-\lambda_{j(\ell+1)} \vert}{\vert \lambda_\ell-\lambda_{j(\ell)} \vert} 
\end{equation*}
is convergent. More precisely, we first choose $\lambda_3$ close to $\lambda_{j(3)}$ (recall that $j(3)\in \{1,2\}$ by definition of $j$) such that 
$$\vert \lambda_3-\lambda_{j(3)}\vert\le \frac{\vert \lambda_2-\lambda_1\vert}{2^3 w_2}$$
and $\lambda_3\notin \{\lambda_1,\lambda_2\}$. Then at step $\ell$ we choose $\lambda_\ell\in \T$ such that
$$
\vert \lambda_{\ell}-\lambda_{j(\ell)}\vert\le \frac{\vert \lambda_{\ell-1}-\lambda_{j(\ell-1)}\vert}{2^\ell w_{\ell-1}}
$$
and $\lambda_\ell\notin \{\lambda_1,\dots, \lambda_{\ell-1}\}$. Under these conditions the series $\mathcal{N}$ is convergent, so that $B$ is a nuclear operator. It follows that $T=D+B$ is a bounded linear operator on $X$.

\noindent $\rhd$ \textbf{Unimodular eigenvectors of the operator $T$.}\\
\noindent Consider the closed subspace
\begin{equation*}
\hat{X_1}:=\overline{\mathrm{span}}\big[x_{1,\ell}\,;\,\ell\ge 1\big]
\end{equation*}
of $X$. Since $Bx_{1,1}=0$ and $Bx_{1,\ell}=\alpha_{1,\ell-1}x_{1,\ell-1}$ for any positive integer $\ell$, $\hat{X_1}$ is $T$-invariant. Hence one can consider the operator
\begin{align*}
T_1: \hat{X_1}&\longrightarrow \hat{X_1}\\
x &\longmapsto T_1x:=Tx
\end{align*}
induced by $T$ on $\hat{X_1}$. Since the decomposition $X=\bigoplus_{\ell\ge 1}X_\ell$ is unconditional, the sequence $(x_{1,\ell})_{\ell \ge 1}$ is an unconditional basis of the space $\hat{X_1}$. Let us now describe the unimodular eigenvectors of the operator $T_1$: if $x=\sum_{\ell\ge 1}c_\ell x_{1,\ell}$ is a vector of $\hat{X_1}$, then the algebraic equation $T_1x=\lambda x$ is satisfied if and only if 
\begin{equation*}
c_\ell=\frac{(\lambda-\lambda_{\ell-1})\dots (\lambda-\lambda_1)}{\alpha_{1,\ell-1}\dots \alpha_{1,1}} c_1\ \ \ \ \ \mathrm{for\ every\ }\ell\ge 2.
\end{equation*}
It follows that for any positive integer $n$, the eigenspace $\mathrm{Ker}(T_1-\lambda_n)$ is $1$-dimensional and that $\mathrm{Ker}(T_1-\lambda_n)=\mathrm{span}\big[u_1^{(n)}\big]$, where
\begin{equation*}
u_1^{(n)}:=x_{1,1}+\sum_{\ell=2}^n\frac{(\lambda_n-\lambda_{\ell-1})\dots (\lambda_n-\lambda_1)}{\alpha_{1,\ell-1}\dots \alpha_{1,1}}x_{1,\ell}.
\end{equation*}
We now need the following fact whose proof can be found in \cite[Theorem 2.1]{EisnerGrivaux}.
\begin{Fac}\emph{(\cite[Lemma 2.4]{EisnerGrivaux})}\label{factspanning}
By choosing in a suitable way the coefficients $w_{1,n}$ and $\lambda_n$, it is possible to ensure that for any integer $n\ge 2$, 
$$
\big\vert\big\vert u_1^{(n)}-u_1^{(j(n))} \big\vert\big\vert\le 2^{-n}.
$$
\end{Fac}
\noindent In order to prove this fact, let us write
\begin{align*}
u_1^{(n)}-u_1^{(j(n))}&=\sum_{\ell=2}^{j(n)}\bigg(\frac{(\lambda_n-\lambda_{\ell-1})\dots (\lambda_n-\lambda_1)}{\alpha_{1,\ell-1}\dots \alpha_{1,1}}-\frac{(\lambda_{j(n)}-\lambda_{\ell-1})\dots (\lambda_{j(n)}-\lambda_1)}{\alpha_{1,\ell-1}\dots \alpha_{1,1}}\bigg)x_{1,\ell}\\
&+\sum_{\ell=j(n)+1}^n\frac{(\lambda_n-\lambda_{\ell-1})\dots (\lambda_n-\lambda_1)}{\alpha_{1,\ell-1}\dots \alpha_{1,1}}x_{1,\ell}:=a_1^{(n)}+b_1^{(n)}
\end{align*}
where we denote the first sum by $a_1^{(n)}$ and the second one by $b_1^{(n)}$. Since the quantities $\alpha_{1,\ell-1}\dots \alpha_{1,1}$ for $\ell\le j(n)$ do not depend from $\lambda_n$, we can ensure that $\vert\vert a_1^{(n)}\vert\vert\le 2^{-(n+1)}$ by taking $\vert \lambda_n-\lambda_{j(n)}\vert$ sufficiently small. Let us now estimate $b_1^{(n)}$. We can assume that $\vert\lambda_p-\lambda_q\vert \le 1$ for any positive integers $p$ and $q$. By using Definitions \ref{poids1} and \ref{poids2} of our weights $\alpha_{i,k}$, we have
\begin{align*}
\vert\vert b_1^{(n)}\vert\vert &\le \sum_{\ell=j(n)+1}^n \bigg\vert \frac{(\lambda_n-\lambda_{\ell-1})\dots (\lambda_n-\lambda_1)}{\alpha_{1,\ell-1}\dots \alpha_{1,1}} \bigg\vert\\
&\le \sum_{\ell=j(n)+1}^n\frac{1}{w_{1,\ell-1}\dots w_{1,1}}\cdot\bigg\vert\frac{\lambda_n-\lambda_{j(n)}}{\lambda_\ell-\lambda_{j(\ell)}}\bigg\vert\cdot
\end{align*}
In order to estimate $\vert\vert b_1^{(n)}\vert\vert$, we choose the coefficients $\lambda_n$ and $w_{1,\ell}$ as follows. At stage $n$ of the construction we take $w_{1,n-1}$ so large with respect to $w_{1,1},\dots, w_{1,n-2}$ so that $(w_{1,n-1}w_{1,n-2}\dots w_{1,1})^{-1}$ is very small (this means that we take $w_{n-1}$ very large by definition \eqref{poids} of $w_{1,n-1}$). After this we take $\lambda_n$ extremely close to $\lambda_{j(n)}$ so that the quantities 
$$
\frac{1}{w_{1,\ell-1}\dots w_{1,1}}\cdot\bigg\vert\frac{\lambda_n-\lambda_{j(n)}}{\lambda_\ell-\lambda_{j(\ell)}}\bigg\vert\qquad (\ell\in \{j(n)+1,\dots, n-1\})
$$
are very small. We can thus ensure that $\vert\vert b_1^{(n)}\vert\vert$ is less than $2^{-(n+1)}$ and the result is proved.

Fact \ref{factspanning} is useful to apply the following theorem which can also be found in \cite{Grivaux}.
\begin{Theo}\emph{(\cite[Theorem 4.2]{Grivaux})}\label{theospanning}
Let $Y$ be a separable infinite-dimensional complex Banach space and let $S$ be a bounded linear operator on $Y$. Suppose that there exists a sequence $(u_i)_{i\ge 1}$ of vectors of $Y$ having the following properties$:$\\
\noindent $(i)$ for each $i\ge 1$, $u_i$ is an eigenvector of $S$ associated to an eigenvalue $\mu_i$ of $S$ where $\vert \mu_i\vert=1$ and the complex numbers $\mu_i$ are all distinct$;$\\
\noindent $(ii)$ $\mathrm{span}\big[u_i\,;\, i\ge 1\big]$ is dense in $Y;$\\
\noindent $(iii)$ for any $i\ge 1$ 	and any $\varepsilon>0$, there exists an $n\ne i$ such that $\vert\vert u_n-u_i\vert\vert<\varepsilon$.\\
\noindent Then $S$ has a perfectly spanning set of eigenvectors associated to unimodular eigenvalues. In particular, the set $\sigma_p(S)\cap \T$ is uncountable.
\end{Theo}
Applying Theorem \ref{theospanning} with $Y:=\hat{X_1}$, $S:=T_1$ and $u_i:=u_1^{(i)}$, we obtain that $T_1$ has a perfectly spanning set of eigenvectors associated to unimodular eigenvalues (the only condition we really need to check in this theorem is condition $(iii)$ and this is true by Fact \ref{factspanning}). Hence the unimodular point spectrum of $T_1$ is uncountable. Since $\sigma_p(T_1)\cap \mathbb{T}\subset\sigma_p(T)\cap\mathbb{T}$, the unimodular point spectrum of $T$ is also uncountable. For this part of the proof, we also refer the reader to the proof of \cite[Proposition 2.5]{EisnerGrivaux}.

\noindent $\rhd$ \textbf{Estimate of the norms $\vert\vert T^{n_k}\vert\vert$.}\\
\noindent As in the \textit{hard part} of the proof of \cite[Theorem 2.1]{EisnerGrivaux}, we now need to prove that if the coefficients $w_n$ and $\lambda_n$ are suitably chosen, $\vert\vert T^{n_k}-D^{n_k}\vert\vert\le 1$ for all positive integers $k$. For all positive integers $k,\ell, i, n$, we put
\begin{equation*}
t_{k,\ell}^{(i,n)}=\langle x_{i,k}^{*},T^nx_{i,\ell}\rangle. 
\end{equation*}
It is then clear that $t_{k,\ell}^{(i,n)}=0$ when $k>\ell$ or $\ell-k>n$. Moreover $t_{k,k}^{(i,n)}=\lambda_k^n$. The lemma below gives the expression of $t_{k,\ell}^{(i,n)}$ for $1\le \ell-k\le n$.
\begin{Lem}\label{expressiont}
For any positive integers $k, \ell,i, n\ge 1$ such that $1\le \ell-k\le n$, we have
\begin{equation*}
t_{k,\ell}^{(i,n)}=\alpha_{i,\ell-1}\alpha_{i,\ell-2}\dots\alpha_{i,k}\sum_{j_k+\dots+ j_{\ell}=n-(\ell-k)}\lambda_k^{j_k}\dots \lambda_\ell^{j_\ell}.
\end{equation*}
\end{Lem}
\begin{proof}
For any fixed positive integer $i$, define the operator $B_i$ on the space 
$$
\hat{X}_i:=\overline{\mathrm{span}}\big[x_{i,\ell}\,;\, \ell\ge 1\big]
$$ 
by setting $B_ix_{i,\ell}:=Bx_{i,\ell}=\alpha_{i,\ell-1}x_{i,\ell-1}$ if $\ell\ge 2$ and $B_i x_{i,1}=0$. Then $B_i$ is a weighted backward shift as in the proof of \cite[Theorem 2.1]{EisnerGrivaux}. Moreover, since the decomposition $X=\bigoplus_{\ell\ge 1}X_\ell$ is unconditional, the sequence $(x_{i,\ell})_{\ell \ge 1}$ is an unconditional Schauder basis of $\hat{X}_i$ and any vector $x$ of $\hat{X}_i$ can be written in a unique way as $x=\sum_{\ell \ge 1}\langle x_{i,\ell}^{*},x\rangle x_{i,\ell}$. It follows that
\begin{equation*}
B_ix=\sum_{\ell \ge 2}\langle x_{i,\ell}^{*},x\rangle\alpha_{i,\ell-1}x_{i,\ell-1}.
\end{equation*}
Then, if we fix the index $i$, we can apply \cite[Lemma 2.6]{EisnerGrivaux} which gives us the expression of the coefficient $t_{k,\ell}^{(i,n)}$.
\end{proof}
\noindent We now want an estimate of $\vert\vert T^{n_p}-D^{n_p}\vert\vert$. In every subspace $X_\ell$ of $X$, we consider the linear subspace 
\begin{equation*}
\tilde{X_\ell}:=\mathrm{span}\big[x_{i,\ell}\,;\, i\ge 1\big].
\end{equation*}
Since $\tilde{X_\ell}$ contains all the vectors $x_{i,\ell}$ ($i\ge 1$), $\tilde{X_\ell}$ is a dense subspace of $X_\ell$. It follows that the direct sum $\tilde{X}:=\bigoplus_{\ell\ge 1}\tilde{X_\ell}$ is dense in $X$. In order to prove our result it suffices to establish that for any vector $x$ of $\tilde{X}$, $\vert\vert (T^{n_p}-D^{n_p})x\vert\vert\le \vert\vert x \vert\vert$. Let $x$ be an element of $\tilde{X}$. Then $x$ can be written as $x=\sum_{\ell\ge 1}x_\ell$ where
\begin{equation*}
x_\ell=\sum_{i=1}^{i_\ell}\langle x_{i,\ell}^{*},y\rangle x_{i,\ell}\in X_\ell
\end{equation*}
where $i_\ell$ is a positive integer for all $\ell\ge 1$. We have then, for any $n\ge 2$,
\begin{align*}
T^n x=\sum_{\ell=1}^{+\infty}\sum_{i=1}^{i_\ell}\langle x_{i,\ell}^{*},x\rangle T^nx_{i,\ell}&=\sum_{\ell=1}^{+\infty}\sum_{i=1}^{i_\ell}\langle x_{i,\ell}^{*},x\rangle \sum_{k=\max(1,\ell-n)}^{\ell}\langle x_{i,k}^{*},T^nx_{i,\ell}\rangle x_{i,k}\\
&=\sum_{\ell=1}^{+\infty}\sum_{i=1}^{i_\ell}\langle x_{i,\ell}^{*},x\rangle \sum_{k=\max(1,\ell-n)}^{\ell}t_{k,\ell}^{(i,n)}x_{i,k}
\end{align*}
and
\begin{equation*}
D^nx=\sum_{\ell=1}^{+\infty}\lambda_\ell^n\sum_{i=1}^{i_\ell}\langle x_{i,\ell}^{*},x\rangle x_{i,\ell}=\sum_{\ell=1}^{+\infty}\sum_{i=1}^{i_\ell}\langle x_{i,\ell}^{*},x\rangle t_{\ell,\ell}^{(i,n)}x_{i,\ell}. 
\end{equation*}
We deduce from this that
\begin{equation*}
(T^n-D^n)x=\sum_{\ell=2}^{+\infty}\sum_{i=1}^{i_\ell}\langle x_{i,\ell}^{*},x\rangle \sum_{k=\max(1,\ell-n)}^{\ell-1}t_{k,\ell}^{(i,n)}x_{i,k}
\end{equation*}
and according to condition \eqref{conditionsysteme}, we get that
\begin{align*}
\vert\vert(T^n-D^n)x\vert\vert &\le \vert\vert x \vert\vert \sum_{\ell=2}^{+\infty}\sum_{i=1}^{i_\ell}\vert\vert x_{i,\ell}^{*}\vert\vert \Bigg(\sum_{k=\max(1,\ell-n)}^{\ell-1}\big\vert t_{k,\ell}^{(i,n)} \big\vert\Bigg)\\
&\le C\vert\vert x \vert\vert \sum_{\ell=2}^{+\infty}M_\ell\Bigg(\sum_{i=1}^{i_\ell}\sum_{k=\max(1,\ell-n)}^{\ell-1}\big\vert t_{k,\ell}^{(i,n)}\big\vert\Bigg).
\end{align*}
We now know from Lemma \ref{expressiont} that
\begin{equation*}
t_{k,\ell}^{(i,n)}=\alpha_{i,\ell-1}\alpha_{i,\ell-2}\dots\alpha_{i,k}s_{k,\ell}^{(n)},
\end{equation*}
where 
$$
s_{k,\ell}^{(n)}:=\sum_{j_k+\dots+ j_{\ell}=n-(\ell-k)}\lambda_k^{j_k}\dots \lambda_\ell^{j_\ell}.
$$
According to the expressions \eqref{poids2} and \eqref{poids} of our coefficients $\alpha_{i,\ell}$ (and $w_{i,\ell}$), this yields that
\begin{equation*}
t_{k,\ell}^{(i,n)}=\frac{w_{\ell-1}w_{\ell-2}\dots w_k}{2^{(\ell-k)i}CM_\ell CM_{\ell-1}\dots CM_{k+1}}\cdot\frac{\vert\lambda_\ell-\lambda_{j(\ell)}\vert}{\vert \lambda_k-\lambda_{j(k)}\vert}\cdot s_{k,\ell}^{(n)}.
\end{equation*}
As $CM_p\ge 1$ for any positive integer $p$, this implies in particular that
\begin{equation*}
\big\vert t_{k,\ell}^{(i,n)}\big\vert\le \frac{w_{\ell-1}w_{\ell-2}\dots w_k}{2^{(\ell-k)i}CM_\ell}\cdot \frac{\vert\lambda_\ell-\lambda_{j(\ell)}\vert}{\vert \lambda_k-\lambda_{j(k)}\vert}\cdot \big\vert s_{k,\ell}^{(n)}\big\vert
\end{equation*}
and then
\begin{equation}\label{densityargument}
\vert\vert (T^{n_p}-D^{n_p})x\vert\vert\le \vert\vert x \vert\vert \sum_{\ell=2}^{+\infty}\sum_{k=\max(1,\ell-n_p)}^{\ell-1}w_{\ell-1}\dots w_k\cdot
\frac{\vert\lambda_\ell-\lambda_{j(\ell)}\vert}{\vert \lambda_k-\lambda_{j(k)}\vert} \cdot\big\vert s_{k,\ell}^{(n_p)}\big\vert.
\end{equation}
It remains to estimate the quantity
\begin{equation}\label{sommepetite}
\sum_{k=\max(1,\ell-n_p)}^{\ell-1}w_{\ell-1}\dots w_k\cdot \frac{\vert\lambda_\ell-\lambda_{j(\ell)}\vert}{\vert \lambda_k-\lambda_{j(k)}\vert} \cdot\big\vert s_{k,\ell}^{(n_p)}\big\vert,
\end{equation}
which is essentially the difficult part of the proof of \cite[Theorem 2.1]{EisnerGrivaux}. It is in this part of the proof where we use the fact that $(n_k)_{k\ge 1}$ is not a Jamison sequence: according to Theorem \ref{jamisonth} this assumption means that for every $\varepsilon>0$ and every $\lambda\in \T$, there exists $\lambda'\in \T\setminus\{\lambda\}$ such that $\vert\lambda-\lambda'\vert\le d_{(n_k)}(\lambda,\lambda')\le \varepsilon$. In their proof, T. Eisner and S. Grivaux prove that if we take $\lambda_\ell$ sufficiently close to $\lambda_{j(\ell)}$ (for the distance $d_{(n_k)}$ hence for the Euclidean norm $\vert\cdot\vert$), then the sum
\begin{equation*}
\sum_{k=\max(1,\ell-n_p)}^{\ell-1}w_{\ell-1}^2\dots w_k^2\cdot \frac{\vert\lambda_\ell-\lambda_{j(\ell)}\vert^2}{\vert \lambda_k-\lambda_{j(k)}\vert^2} \cdot\big\vert s_{k,\ell}^{(n_p)}\big\vert^2
\end{equation*}
can be made arbitrarily small. By rewriting the proof, one can see that if we take in the same way $\lambda_\ell$ close to $\lambda_{j(\ell)}$, the sum \eqref{sommepetite} is less than $2^{1-\ell}$. By density, we deduce from \eqref{densityargument} that $\vert\vert T^{n_p}-D^{n_p}\vert\vert\le 1$ for any positive integer $p$, which concludes the proof of the theorem.
\end{proof}

Example \ref{fonctionscontinues} allows us to give new examples of universal Jamison spaces.

\subsection{Examples}

Recall that the James space $\mathcal{J}$ is the set of all complex sequences $x=(x_n)_{n\ge 1}$ belonging the space $c_0(\N)$ such that \begin{equation*}
\vert\vert x \vert\vert_\mathcal{J}:=\sup \Big\{\vert x_{p_1}-x_{p_2}\vert^2+\dots +\vert x_{p_{k-1}}-x_{p_k}\vert^2 \Big\}<+\infty,
\end{equation*}
where the supremum is taken over all the $k$-tuples $(p_1,\dots,p_k)$ of positive integers such that $p_1<\dots<p_k$. We refer the reader to the book \cite{Fetter} for more information on the James space. In particular, it is not difficult to see that the subspace 
\begin{equation*}
\mathcal{J}_2=\big\{ x\in \mathcal{J}\,;\, x_{2k}=0\ \mathrm{for\ all\ }k\ge 1 \big\}
\end{equation*}
of $\mathcal{J}$ is linearly isomorphic to the Hilbert space $\ell_2(\N)$ and that this subspace is complemented in $\mathcal{J}$. Then Example \ref{fonctionscontinues} gives us the following result.

\begin{Ex}
The James space $\mathcal{J}$ is a universal Jamison space.
\end{Ex}

Other examples of universal Jamison spaces are given by spaces which contain a copy of $c_0(\N)$. For instance, the space $C([0,1])$ is a universal Jamison space. On the other hand, it is not difficult to exhibit a class of Banach spaces which are not universal Jamison spaces, namely the class of \textit{hereditarily indecomposable Banach spaces}.

\begin{Def}
An infinite-dimensional Banach space $X$ is said to be \textit{decomposable} if there exists two infinite-dimensional closed subspaces $Y$ and $Z$ of $X$ such that $X=Y\oplus Z$. We say that $X$ is \textit{hereditarily indecomposable} if no infinite-dimensional closed subspace of $X$ is decomposable.
\end{Def}

The famous Gowers dichotomy highlights the fact that the notion of unconditional Schauder decomposition is in a sense opposite to the notion of hereditarily indecomposable space. More precisely, Gowers \cite{Gowers} proves that if $X$ is an arbitrary Banach space it either contains an unconditional basic sequence or contains a hereditarily indecomposable subspace. In fact, a hereditarily indecomposable complex Banach space $X$ fails to be a universal Jamison space. Indeed, a bounded linear operator $T$ on this space $X$ is of the form $\lambda I+S$ where $\lambda\in \C$ and $S$ is a strictly singular operator, that is an operator $S\in \mathcal{B}(X)$ which fails to be an isomorphism when restricted to any infinite-dimensional closed subspace of $X$. In particular, the unimodular point spectrum of $T$ is at most countable (see \cite{Maurey} for more details). We have thus proved:

\begin{Prop}
If $X$ is a hereditarily indecomposable Banach space, then $X$ is not a universal Jamison space.
\end{Prop}

\section{Jamison sequences for $C_0$-semigroups}
In this section, we study Jamison sequences in the context of strongly continuous semigroups acting on separable complex Banach spaces. Starting from the work of C. Badea and S. Grivaux (see \cite{BadeaGrivaux1} and \cite{BadeaGrivaux2}), we give a characterization of Jamison sequences for $C_0$-semigroups (Definition \ref{jamisonsemigroupes} below). \\
We begin by recalling some definitions and facts about $C_0$-semigroups. Let $X$ be a complex Banach space. A family of bounded linear operators $(T_t)_{t\ge 0}$ on $X$ is called a $C_0$-\textit{semigroup} if\\
\noindent $\bullet$ $T_0=Id_X$;\\
\noindent $\bullet$ for any $s,t\ge 0$, $T_{s+t}=T_s T_t$;\\
\noindent $\bullet$ for any $x\in X$, $\underset{t\to 0^+}{\lim}\vert\vert T_tx-x\vert\vert=0$.\\
\noindent The \textit{infinitesimal generator} of the $C_0$-semigroup $(T_t)_{t\ge 0}$ is the map $A: D(A) \longrightarrow X$ defined by
\begin{align*}
D(A)&:=\bigg\{ x\in X\,;\, \underset{t\to 0^{+}}{\lim}\frac{T_tx-x}{t}\ \mathrm{exists}\bigg\},\\
Ax&:=\underset{t\to 0^{+}}{\lim}\frac{T_tx-x}{t}\ \ \ (x\in D(A)).
\end{align*}
We recall that $\sigma_p(T_t)\setminus\{0\}=\exp(t\sigma_p(A))$ for any $t\ge 0$ (see for instance \cite[Chapter 4]{EngelNagel}) and that
\begin{equation*}
\sigma_p(T_1)\cap \mathbb{T}=\exp(\sigma_p(A)\cap i\R).
\end{equation*}

We now introduce the notion of \textit{semigroup partially bounded with respect to some sequence of positive real numbers}.

\begin{Def}
Let $X$ be a separable complex Banach space. Let $(t_k)_{k\ge 0}$ be an increasing sequence of positive real numbers and $(T_t)_{t\ge 0}$ a semigroup of bounded linear operators on $X$. We say that $(T_t)_{t\ge 0}$ is \textit{partially bounded with respect to the sequence} $(t_k)_{k\ge 0}$ if $\sup_{k\ge 0}\vert\vert T_{t_k}\vert\vert<+\infty$.
\end{Def}

When the sequence of positive real numbers $(t_k)_{k\ge 0}$ is bounded, then any $C_0$-semigroup is bounded with respect to the sequence $(t_k)_{k\ge 0}$. So we restrict ourselves in the sequel to increasing sequences $(t_k)_{k\ge 0}$ of positive real numbers which tend to infinity. Moreover, it is clear that we always can 	make the assumption that $t_0=1$.

\begin{Def}\label{jamisonsemigroupes}
Let $(t_k)_{k\ge 0}$ be an increasing sequence of positive real numbers which tends to infinity. We say that $(t_k)_{k\ge 0}$ is a \textit{Jamison sequence for} $C_0$-\textit{semigroups} if for every separable complex Banach space $X$ and for every $C_0$-semigroup $(T_t)_{t\ge 0}$ of bounded linear operators on $X$ (with infinitesimal generator $A$) which is partially bounded with respect to the sequence $(t_k)_{k\ge 0}$, the set $\sigma_p(A)\cap i\R$ is at most countable.
\end{Def}

Our aim in this section is to give a characterization of Jamison sequences for $C_0$-semigroups. In order to do this, we will need the characterization of Jamison sequences which was obtained in \cite[Theorem 2.1]{BadeaGrivaux2}. Let us introduce the function $\vert\vert \cdot\vert\vert$ defined by
\begin{equation*}
\vert\vert \theta\vert\vert :=\mathrm{dist}(\theta,\Z)=\inf\big\{\vert \theta-n\vert\,;\, n\in\Z \big\}.
\end{equation*}
Recall that there exists two constants $C_1>0$ and $C_2>0$ such that
\begin{equation*}
C_2\vert\vert \theta\vert\vert\le \big\vert e^{2i\pi\theta}-1\big\vert\le C_1\vert\vert \theta \vert\vert.
\end{equation*}
for any real number $\theta$. We are going to prove the following result.

\begin{Theo}\label{caracterisationjamisonsemigroupes}
Let $(t_k)_{k\ge 0}$ be an increasing sequence of positive real numbers such that $t_0=1$ and $t_k\longrightarrow +\infty$. The following assertions are equivalent$:$\\
\noindent $(1)$ the sequence $(t_k)_{k\ge 0}$ is a Jamison sequence for $C_0$-semigroups$;$\\
\noindent $(2)$ there exists $\varepsilon>0$ such that for any $\theta\in \big]0,\frac{1}{2}\big]$,
 \begin{equation*}
\sup_{k\ge 0}\vert\vert t_k\theta\vert\vert\ge \varepsilon.
\end{equation*}
\end{Theo}

In order to be able to use the characterization of Jamison sequences of \cite[Theorem 2.1]{BadeaGrivaux2}, we only consider at the beginning \textit{sequences of positive integers} which are Jamison sequences for $C_0$-semigroups. We show first that there is a relationship between the notion of Jamison sequence and that of Jamison sequence for $C_0$-semigroups when we restrict ourselves to sequences of integers (Theorem \ref{jamisoncontinu}). Once this result is proved, we will be able to prove Theorem \ref{caracterisationjamisonsemigroupes}.

\subsection{Sequences of integers which are Jamison sequences for $C_0$-semigroups}
Let $(n_k)_{k\ge 0}$ be an increasing sequence of positive integers such that $n_0=1$. As in the proof of \cite[Theorem 2.8]{BadeaGrivaux2}, we associate to this sequence a distance $d_{(n_k)}$ on the unit circle $\T$ by setting
\begin{equation*}
d_{(n_k)}(\lambda,\mu)=\sup_{k\ge 0}\big\vert \lambda^{n_k}-\mu^{n_k}\big\vert\qquad \mathrm{for\ any\ }\lambda, \mu \in \T.
\end{equation*}

We now prove the theorem below.

\begin{Theo}\label{jamisoncontinu}
Let $(n_k)_{k\ge 0}$ be an increasing sequence of positive integers such that $n_0=1$. The following assertions are equivalent$:$\\
\noindent $(1)$ the sequence $(n_k)_{k\ge 0}$ is a Jamison sequence for $C_0$-semigroups$;$\\
\noindent $(2)$ the sequence $(n_k)_{k\ge 0}$ is a Jamison sequence$;$\\
\noindent $(3)$ there exists $\varepsilon>0$ such that for any two distinct points $\lambda, \mu\in \T$,
\begin{equation*}
d_{(n_k)}(\lambda,\mu)\ge \varepsilon.
\end{equation*}
\end{Theo}
\begin{proof}
According to \cite[Theorem 2.1]{BadeaGrivaux2}, assertions $(2)$ and $(3)$ are equivalent. So let us prove $(3) \Rightarrow (1)$. Let $X$ be a separable complex Banach space and $(T_t)_{t\ge 0}$ a $C_0$-semigroup of bounded linear operators on $X$ such that $M:=\sup_{k\ge 0}\vert\vert T_{n_k}\vert\vert<+\infty$. We denote by $A$ the infinitesimal generator of $(T_t)_{t\ge 0}$. Let $i\eta$ and $i\xi$ be two eigenvalues of $A$ ($\eta,\xi\in\R$) and let $x_\eta$ and $x_\xi$ be eigenvectors such that $\vert\vert x_\eta\vert\vert=\vert\vert x_\xi\vert\vert=1$ and $Ax_\phi =i\phi x_\phi$ for $\phi\in\{\eta,\xi\}$. We know that for any nonnegative integer $k$ and for $\phi\in\{\eta,\xi\}$, we have $T_{n_k}x_\phi =e^{i\phi n_k}x_\phi$. By the triangle inequality,
\begin{equation}\label{inegalitestriangulaires}
\big\vert e^{i\eta n_k}-e^{i\xi n_k}\big\vert-\vert\vert x_\eta-x_\xi \vert\vert\le \vert\vert T_{n_k}(x_\eta-x_\xi)\vert\vert\le M\vert\vert x_\eta-x_\xi \vert\vert.
\end{equation}
Setting $\lambda_\theta=e^{i\theta}$ for $\theta\in \R$, we deduce from \eqref{inegalitestriangulaires} that
\begin{equation*}
\vert\vert x_\eta-x_\xi \vert\vert\ge \frac{\sup_{k\ge 0}\big\vert \lambda_{\eta-\xi}^{n_k}-1 \big\vert}{M+1}\cdot
\end{equation*}
Under the assumption that $\eta,\xi \in [2\ell\pi,2(\ell+1)\pi[$ for some integer $\ell$ and that $\eta\ne \xi$, $\lambda_{\eta-\xi}\in\mathbb{T}\setminus\{1\}$ and $\sup_{k\ge 0}\big\vert \lambda_{\eta-\xi}^{n_k}-1 \big\vert\ge \varepsilon$. It follows that 
\begin{equation*}
\vert\vert x_\eta-x_\xi \vert\vert\ge \frac{\varepsilon}{M+1}\cdot
\end{equation*}
Since the space $X$ is separable, we thus obtain that the set $\sigma_p(A)\cap [2i\ell\pi,2i(\ell+1)\pi[$ is at most countable. We then conclude that the set $\sigma_p(A)\cap i\mathbb{R}$ itself is at most countable, which proves that $(t_k)_{k\ge 0}$ is a Jamison sequence for $C_0$-semigroups.\\
\noindent We now prove $(1) \Rightarrow (3)$. Assume that property $(3)$ is not satisfied. We know from \cite[Theorem 2.8]{BadeaGrivaux2} that there exists an uncountable subset $K$ of $\T$ such that the metric space $(K,d_{(n_k)})$ is separable and we want to prove that $(n_k)_{k\ge 0}$ is not a Jamison sequence for $C_0$-semigroups. In order to do this, we construct a separable Banach space $X$ and a $C_0$-semigroup $(S_t)_{t\ge 0}$ on $X$ (with infinitesimal generator $A$) which is bounded with respect to the sequence $(n_k)_{k\ge 0}$ and such that the set $\sigma_p(A)\cap i\R$ is uncountable. Let
\begin{equation*}
X=\Bigg\{ f:[0,+\infty[\longrightarrow \R\ \mathrm{measurable}\,;\, \vert\vert f \vert\vert := \bigg(\int_0^{+\infty}\frac{\vert f(t)\vert^2}{1+t^2}\,dt\bigg)^{1/2}<+\infty \Bigg\}
\end{equation*}
and let $(S_t)_{t\ge 0}$ be the translation semigroup on $X$ defined by 
\begin{equation*}
S_t f(x)=f(x+t)\ \ \ \ \ \ \ \ (f\in X,\, t,x\ge 0).
\end{equation*}
We introduce a new space associated to the semigroup $(S_t)_{t\ge 0}$ by setting $X_{*}:=\big\{f\in X\,;\,\vert\vert f \vert\vert_{*}<+\infty \big\}$ where the norm $\vert\vert\cdot\vert\vert_{*}$ is defined by
\begin{equation*}
\vert\vert f \vert\vert_{*}:=\max\Bigg( \vert\vert f \vert\vert,\, \sup_{j\ge 0}\,4^{-(j+1)}\sup_{k_0,\dots, k_j\ge 0}\Bigg\vert\Bigg\vert \prod_{\ell=0}^j(S_{n_{k_{\ell}}} -I)f\Bigg\vert\Bigg\vert \Bigg).
\end{equation*}
It is rather easy to check that $S_t$ is a bounded linear operator on $X_{*}$. In a first step, we prove that the semigroup $(S_t)_{t\ge 0}$ is bounded with respect to the sequence $(n_k)_{k\ge 0}$. Then we construct a separable subspace of $X_*$ such that the semigroup $(S_t)_{t\ge 0}$ (defined on this space) is strongly continuous and such that the set  $\sigma_p(A)\cap i\mathbb{R}$ is uncountable, where $A$ denotes the infinitesimal generator of $(S_t)_{t\ge 0}$.

\vspace*{0.5cm}

\noindent $\rhd$ \textbf{Boundedness of the semigroup $(S_t)_{t\ge 0}$ with respect to $(n_k)_{k\ge 0}$.}\\
\noindent Let $f\in X_*$ and $k\in \Z_+$. The norm $\vert\vert S_{n_k}f\vert\vert_{*}$ is equal to the maximum between the quantities
\begin{equation*}
\vert\vert S_{n_k}f\vert\vert\ \ \ \ \ \mathrm{and}\ \ \ \ \ \sup_{j\ge 0}\, 4^{-(j+1)}\sup_{k_0,\dots, k_j\ge 0}\Bigg\vert\Bigg\vert \prod_{\ell=0}^j (S_{n_{k_\ell}}-I)S_{n_k}f \Bigg\vert\Bigg\vert.
\end{equation*}
On the one hand,
\begin{align*}
\vert\vert S_{n_k}f\vert\vert=\vert\vert f+(S_{n_k}-I)f \vert\vert &\le \vert\vert f \vert\vert+\vert\vert (S_{n_k}-I)f\vert\vert\\
&\le \vert\vert f \vert\vert_{*}+4\cdot\frac{1}{4}\vert\vert (S_{n_k}-I)f \vert\vert\\
&\le 5\vert\vert f \vert\vert_{*}
\end{align*}
and on the other hand, for any $j\in \Z_+$ and any $(j+1)$-tuple $(k_0,\dots, k_j)$ of nonnegative integers, we have
\begin{align*}
4^{-(j+1)}\Bigg\vert\Bigg\vert \prod_{\ell=0}^j(S_{n_{k_{\ell}}} -I)S_{n_k}f\Bigg\vert\Bigg\vert\le 4&\cdot 4^{-(j+2)}\Bigg\vert\Bigg\vert \prod_{\ell=0}^j (S_{n_{k_\ell}}-I)(S_{n_k}-I)f \Bigg\vert\Bigg\vert\\
&+4^{-(j+1)}\Bigg\vert\Bigg\vert \prod_{\ell=0}^j (S_{n_{k_\ell}}-I)f \Bigg\vert\Bigg\vert.
\end{align*}
This proves that
\begin{equation*}
\vert\vert S_{n_k}f\vert\vert_{*}\le 4\vert\vert f \vert\vert_{*}+\vert\vert f \vert\vert_{*}=5\vert\vert f \vert\vert_{*}
\end{equation*}
and so $\sup_{k\ge 0}\vert\vert S_{n_k}\vert\vert_{*}\le 5$.

\vspace*{0.5cm}

\noindent $\rhd$ \textbf{Eigenvectors of the infinitesimal generator $A$.}\\
\noindent For any $\eta\in [0,2\pi[$, we put $e_\eta(x)=e^{i\eta x}$ $(x\in \R)$. It is clear that $e_\eta$ is an eigenvector of $A$ associated to the eigenvalue $i\eta$ and that $e_\eta$ belongs to the space $X_*$. Indeed, $\vert\vert e_\eta \vert\vert =\sqrt{\frac{\pi}{2}}$ and for any nonnegative integer $j$ and any $(j+1)$-tuple $(k_0,\dots, k_j)$ of nonnegative integers, we have
\begin{equation*}
\Bigg\vert\Bigg\vert \prod_{\ell=0}^j (S_{n_{k_\ell}}-I)e_\eta \Bigg\vert\Bigg\vert=\Bigg(\prod_{\ell=0}^j \big\vert e^{i\eta n_{k_\ell}}-1 \big\vert\Bigg)\sqrt{\frac{\pi}{2}}\le 2^{j+1}\sqrt{\frac{\pi}{2}},
\end{equation*}
which proves that $\vert\vert e_\eta \vert\vert_{*}=\sqrt{\frac{\pi}{2}}$.

\vspace*{0.5cm}

\noindent $\rhd$ \textbf{Making the space $X_*$ separable.}\\
\noindent At this stage of the proof, we use our assumption: we know that there exists an uncountable subset $K$ of $\T$ such that the metric space $(K,d_{(n_k)})$ is separable. We also define the set 
\begin{equation*}
I_K:=\big\{\eta\in [0,2\pi[\,;\, e^{i\eta}\in K\big\}.
\end{equation*}
The subspace of $X_*$ we are going to consider in the sequel of the proof is
\begin{equation*}
X_{*}^K:=\overline{\mathrm{span}}^{\vert\vert\cdot \vert\vert_{*}}\big[ e_\eta\,;\, \eta\in I_K \big]
\end{equation*}
equipped with the norm $\vert\vert\cdot\vert\vert_*$. Since the set $\sigma_p(A)\cap i\mathbb{R}$ contains $I_K$, it is uncountable. Furthermore, the semigroup $(S_t)_{t\ge 0}$ is still bounded with respect to the sequence $(n_k)_{k\ge 0}$ and it is easy to prove that it is also strongly continuous by using a density argument. The only thing we really need to prove is that the space $X_*^K$ is separable. This is a consequence of the following lemma.

\begin{Lem}
The eigenvector field $E : I_K\longrightarrow X_{*}$ defined by $E(\eta)=e_\eta$ is continuous on $I_K$.
\end{Lem}
\begin{proof}
Let $\eta,\mu \in I_K$. We need to estimate the quantity $\vert\vert e_\eta-e_\xi\vert\vert_{*}$, which is equal to
\begin{equation*}
\max\Bigg( \vert\vert e_\eta-e_\xi\vert\vert,\, \sup_{j\ge 0}\, 4^{-(j+1)}\sup_{k_0,\dots, k_j\ge 0}\Bigg\vert\Bigg\vert \prod_{\ell=0}^j \big(e^{i\eta n_{k_\ell}}-1\big)e_\eta-\prod_{\ell=0}^j\big(e^{i\xi n_{k_\ell}}-1\big)e_\xi \Bigg\vert\Bigg\vert \Bigg).
\end{equation*}
For any nonnegative integer $j$ and for any $(j+1)$-tuple $(k_0,\dots, k_j)\in (\mathbb{Z}_+)^{j+1}$, we have
\begin{align*}
\Bigg\vert\Bigg\vert \prod_{\ell=0}^j \big(e^{i\eta n_{k_\ell}}-1\big)e_\eta-\prod_{\ell=0}^j\big(e^{i\xi n_{k_\ell}}-1\big)e_\xi \Bigg\vert\Bigg\vert &\le \vert\vert  e_\eta-e_\xi\vert\vert \prod_{\ell=0}^j \big\vert e^{i\eta n_{k_\ell}}-1\big\vert\\
&+\vert\vert e_\xi\vert\vert \Bigg\vert \prod_{\ell=0}^j \big(e^{i\eta n_{k_\ell}}-1\big)-\prod_{\ell=0}^j\big(e^{i\xi n_{k_\ell}}-1\big) \Bigg\vert
\end{align*}
and then
\begin{align}\label{estimationequation2}
\notag \Bigg\vert\Bigg\vert \prod_{\ell=0}^j \big(e^{i\eta n_{k_\ell}}-1\big)e_\eta-\prod_{\ell=0}^j\big(e^{i\xi n_{k_\ell}}-1 &\big)e_\xi \Bigg\vert\Bigg\vert \le  2^{j+1}\vert\vert 
 e_\eta-e_\xi\vert\vert\\
&+\sqrt{\frac{\pi}{2}}\Bigg\vert \prod_{\ell=0}^j \big(e^{i\eta n_{k_\ell}}-1\big)-\prod_{\ell=0}^j\big(e^{i\xi n_{k_\ell}}-1\big) \Bigg\vert.
\end{align}
We need to estimate the quantity which appears in \eqref{estimationequation2}. For this, we write for $j\ge 0$,
\begin{equation*}
d_j\big(e^{i\eta},e^{i\xi}\big)=\sup_{k_0,\dots, k_j\ge 0}\Bigg\vert \prod_{\ell=0}^j \big(e^{i\eta n_{k_\ell}}-1)-\prod_{\ell=0}^j\big(  e^{i\xi n_{k_\ell}}-1\big) \Bigg\vert.
\end{equation*}
According to the identity
\begin{align*}
\prod_{\ell=0}^j \big(e^{i\eta n_{k_\ell}}-1\big)-\prod_{\ell=0}^j\big(e^{i\xi n_{k_\ell}}-1\big)=\big(& e^{i\eta n_{k_0}}-e^{i\xi t_{k_0}} \big)\prod_{\ell=1}^j\big( e^{i\eta n_{k_\ell}}-1 \big)\\
&+\big(e^{i\xi n_{k_0}}-1\big)\Bigg(\prod_{\ell=1}^j\big(e^{i\eta n_{k_\ell}}-1\big)-\prod_{\ell=1}^j\big(e^{i\xi n_{k_\ell}}-1\big) \Bigg),
\end{align*}
we get the estimate 
\begin{equation*}
d_j\big(e^{i\eta},e^{i\xi}\big)\le 2^j d_{(n_k)}\big(e^{i\eta},e^{i\xi}\big)+2d_{j-1}\big(e^{i\eta},e^{i\xi}\big).
\end{equation*}
Then it follows from an easy induction argument that 
\begin{equation*}
d_j\big(e^{i\eta}, e^{i\xi}\big)\le (j+1)2^j d_{(n_k)}\big(e^{i\eta},e^{i\xi}\big)
\end{equation*}
for any nonnegative integer $j$. The above estimates give us
\begin{equation*}
\vert\vert e_\eta-e_\xi\vert\vert_{*}\le\max\bigg( \vert\vert e_\eta-e_\xi\vert\vert, \sup_{j\ge 0}\Big( 2^{-(j+1)}\vert\vert e_\eta-e_\xi\vert\vert + \sqrt{\frac{\pi}{2}}(j+1)2^{-(j+2)}d_{(n_k)}\big(e^{i\eta},e^{i\xi}\big) \Big) \bigg).
\end{equation*}
Then there exists a constant $C>0$ such that for any $\eta,\mu\in I_K$,
\begin{equation*}
\vert\vert e_\eta-e_\xi\vert\vert_{*}\le C\big( \vert\vert e_\eta-e_\xi\vert\vert+ d_{(n_k)}\big(e^{i\eta},e^{i\xi}\big) \big).
\end{equation*}
We now fix $\varepsilon>0$. There exists $A_\varepsilon>0$ such that
\begin{equation*}
\int_{A_\epsilon}^{+\infty}\frac{\big\vert e_\eta(t)-e_\xi(t) \big\vert^2}{1+t^2}\,dt\le \int_{A_\epsilon}^{+\infty}\frac{4}{1+t^2}\,dt\le \varepsilon^2
\end{equation*}
and then
\begin{equation*}
\vert\vert e_\eta-e_\xi\vert\vert_{*}\le C\Bigg(\Bigg(\int_0^{A_\epsilon}\frac{\big\vert e^{i\eta t}-e^{i\xi t}\big\vert^2}{1+t^2}\,dt\Bigg)^{1/2}+d_{(n_k)}\big(e^{i\eta},e^{i\xi}\big)+\varepsilon\Bigg).
\end{equation*}
It follows that the eigenvector field $E:\eta\longmapsto e_\eta$ is continuous on $I_K$.
\end{proof}
As a consequence, the space $X_*^K$ is separable, which proves the implication $(1) \Rightarrow (3)$.
\end{proof}

With the help of Theorem \ref{jamisoncontinu}, we will be able to prove our result on the characterization of Jamison sequences for $C_0$-semigroups (Theorem \ref{caracterisationjamisonsemigroupes}).

\subsection{Relationship with real Jamison sequences for $C_0$-semigroups}
Let $(t_k)_{k\ge 0}$ be an increasing sequence of positive real numbers such that $t_0=1$ and $t_k\longrightarrow +\infty$. For any nonnegative integer $k$, let us denote by $n_k$ the integer part of $t_k$. In particular, $n_0=1$. To begin with, we have the following easy fact.

\begin{Fac}\label{faitjamison}
If $(T_t)_{t\ge 0}$ is a $C_0$-semigroup of bounded linear operators on a Banach space $X$ such that $\sup_{k\ge 0}\vert\vert T_{n_k}\vert\vert<+\infty$, then $\sup_{k\ge 0}\vert\vert T_{t_k}\vert\vert<+\infty$.
\end{Fac}

\begin{proof}
The proof is a consequence of the fact that the family $\big\{T_s\,;\, s\in [0,1]\big\}$ is bounded. Indeed, for any nonnegative integer $k$, $\varepsilon_k=t_k-n_k\in [0,1[$ and so
\begin{equation*}
\vert\vert T_{t_k}\vert\vert=\vert\vert T_{n_k}T_{\varepsilon_k}\vert\vert\le \vert\vert T_{n_k}\vert\vert\,\vert\vert T_{\varepsilon_k}\vert\vert\le \bigg(\sup_{0\le s\le 1}\vert\vert T_s\vert\vert\bigg)\vert\vert T_{n_k}\vert\vert
\end{equation*}
and the conclusion follows from the assumption.
\end{proof}

The characterization of Jamison sequences for $C_0$-semigroups is a consequence of the two following lemmas.

\begin{Lem}\label{lemme1jamison}
\noindent $(1)$ If $(t_k)_{k\ge 0}$ is a Jamison sequence for $C_0$-semigroups, then $(n_k)_{k\ge 0}$ is a Jamison sequence for $C_0$-semigroups as well.\\
\noindent $(2)$ If $\big( 1,(n_k+1)_{k\ge 0} \big)$ is a Jamison sequence for $C_0$-semigroups, then $(t_k)_{k\ge 0}$ is a Jamison sequence for $C_0$-semigroups as well.
\end{Lem}

\begin{proof}
Assume that $(t_k)_{k\ge 0}$ is a Jamison sequence for $C_0$-semigroups. Let $(T_t)_{t\ge 0}$ be a $C_0$-semigroup (with infinitesimal generator $A$) of bounded linear operators on a separable complex Banach space $X$ such that $\sup_{k\ge 0}\vert\vert T_{n_k}\vert\vert<+\infty$. According to Fact \ref{faitjamison}, we know that $\sup_{k\ge 0}\vert\vert T_{t_k}\vert\vert<+\infty$ and it follows from the assumption that the set $\sigma_p(A)\cap i\R$ is at most countable. Then $(n_k)_{k\ge 0}$ is a Jamison sequence for $C_0$-semigroups and $(1)$ is proved. We prove assertion $(2)$ with the same method: assume that the sequence $\big( 1,(n_k+1)_{k\ge 0} \big)$ is a Jamison sequence for $C_0$-semigroups and let $(T_t)_{t\ge 0}$ be a $C_0$-semigroup (with infinitesimal generator $A$) of bounded linear operators on a separable complex Banach space $X$ such that $\sup_{k\ge 0}\vert\vert T_{t_k}\vert\vert<+\infty$. By the definition of $n_k$, the quantity $\varepsilon_k=n_k+1-t_k$ belongs to $]0,1]$ and the same proof as that of Fact \ref{faitjamison} shows that $\sup_{k\ge 0}\vert\vert T_{n_k+1}\vert\vert<+\infty$. Since $\big( 1,(n_k+1)_{k\ge 0} \big)$ is a Jamison sequence for $C_0$-semigroups, the set $\sigma_p(A)\cap i\R$ is at most countable and then the sequence $(t_k)_{k\ge 0}$ is a Jamison sequence for $C_0$-semigroups.
\end{proof}

In the second lemma, we establish a relationship between the sequences of integers $(n_k)_{k\ge 0}$ and $\big(1,(n_k+1)_{k\ge 0}\big)$ from the point of view of Jamison sequences.

\begin{Lem}\label{lemme2jamison}
The sequence $(n_k)_{k\ge 0}$ is a Jamison sequence if and only if $\big(1,(n_k+1)_{k\ge 0}\big)$ is a Jamison sequence.
\end{Lem}

\begin{proof}
Assume that $(n_k)_{k\ge 0}$ is a Jamison sequence. According to \cite[Theorem 2.1]{BadeaGrivaux2}, there exists $\varepsilon>0$ such that for any $\lambda \in \T\setminus \{1\}$, we have $\sup_{k\ge 0}\big\vert \lambda^{n_k}-1 \big\vert\ge \varepsilon$. Let $\lambda\in \mathbb{T}\setminus\{1\}$ such that $\big\vert \lambda-1\big\vert\le \frac{\varepsilon}{2}$. For any nonnegative integer $k$, we get
\begin{equation*}
\big\vert \lambda^{n_k+1}-1 \big\vert=\big\vert \lambda\big( \lambda^{n_k}-1 \big)+\lambda-1 \big\vert\ge \big\vert\lambda^{n_k}-1\big\vert-\frac{\varepsilon}{2}
\end{equation*}
and then
\begin{equation*}
\sup_{k\ge 0}\big\vert \lambda^{n_k+1}-1 \big\vert\ge \frac{\varepsilon}{2}\cdot
\end{equation*}
It follows that
\begin{equation*}
\max\Bigg( \big\vert \lambda-1\big\vert, \sup_{k\ge 0}\big\vert \lambda^{n_k+1}-1 \big\vert \Bigg)\ge \frac{\varepsilon}{2}\cdot
\end{equation*}
According to \cite[Theorem 2.1]{BadeaGrivaux2}, the sequence $\big(1,(n_k+1)_{k\ge 0}\big)$ is a Jamison sequence. On the other hand, if $\big(1,(n_k+1)_{k\ge 0}\big)$ is a Jamison sequence then it is straightforward to check that an operator which is partially power-bounded with respect to the sequence $(n_k)_{k\ge 0}$ is also partially power-bounded with respect to the sequence $\big(1,(n_k+1)_{k\ge 0}\big)$. This proves that $(n_k)_{k\ge 0}$ is also a Jamison sequence.
\end{proof}

The above lemmas and Theorem \ref{jamisoncontinu} allow us to prove the following result.

\begin{Theo}\label{equivalencejamison}
Let $(t_k)_{k\ge 0}$ be an increasing sequence of positive real numbers such that $t_0=1$ and $t_k \longrightarrow +\infty$. For any nonnegative integer $k$, we denote by $n_k$ the integer part of $t_k$. The following assertions are equivalent$:$\\
\noindent $(1)$ the sequence $(t_k)_{k\ge 0}$ is a Jamison sequence for $C_0$-semigroups$;$\\
\noindent $(2)$ the sequence $(n_k)_{k\ge 0}$ is a Jamison sequence$;$\\
\noindent $(3)$ there exists $\varepsilon>0$ such that for any $\lambda\in \T\setminus\{1\}$, we have $\sup_{k\ge 0}\big\vert \lambda^{n_k}-1 \big\vert\ge \varepsilon$. 
\end{Theo}

\begin{proof}
We already know from \cite[Theorem 2.1]{BadeaGrivaux2} that assertions $(2)$ and $(3)$ are equivalent. We now prove $(1) \Rightarrow (2)$. Assume that $(t_k)_{k\ge 0}$ is a Jamison sequence for $C_0$-semigroups. Applying Lemma \ref{lemme1jamison} and Theorem \ref{jamisoncontinu}, we obtain that $(n_k)_{k\ge 0}$ is a Jamison sequence and so $(1) \Rightarrow (2)$. We then prove $(2) \Rightarrow (1)$. If $(n_k)_{k\ge 0}$ is a Jamison sequence, then we know from Lemma \ref{lemme2jamison} that $\big(1,(n_k+1)_{k\ge 0}\big)$ is a Jamison sequence. According to Theorem \ref{jamisoncontinu}, the sequence $\big(1,(n_k+1)_{k\ge 0}\big)$ is a Jamison sequence for $C_0$-semigroups. It remains to apply Lemma \ref{lemme1jamison} in order to obtain that $(t_k)_{k\ge 0}$ is a Jamison sequence for $C_0$-semigroups, and this proves $(2) \Rightarrow (1)$.
\end{proof}

Theorem \ref{caracterisationjamisonsemigroupes} is a consequence of Theorem \ref{equivalencejamison} and the next proposition.

\begin{Prop}
Let $(t_k)_{k\ge 0}$ be an increasing sequence of positive real numbers such that $t_0=1$ and $t_k \longrightarrow +\infty$. For any nonnegative integer $k$, let us denote by $n_k$ the integer part of $t_k$. The following assertions are equivalent$:$\\
\noindent $(i)$ there exists $\varepsilon>0$ such that for any $\lambda\in\T\setminus\{1\}$, we have $\sup_{k\ge 0}\big\vert \lambda^{n_k}-1\big\vert\ge \varepsilon;$\\ 
\noindent $(ii)$ there exists $\varepsilon'>0$ such that for any $\theta\in \big]0,\frac{1}{2}\big]$, we have $\sup_{k\ge 0}\vert\vert t_k\theta\vert\vert\ge \varepsilon'$.
\end{Prop}

\begin{proof}
Setting $\lambda=e^{2i\pi\theta}$ for $\theta\in \big]0,\frac{1}{2}\big]$, we know that $C_2\vert\vert\theta\vert\vert\le \vert \lambda_\theta-1\vert\le C_1\vert\vert\theta\vert\vert$ where $C_1$ and $C_2$ are positive constants which not depend on $\theta$. By using these inequalities, it is very easy to check that assertions $(i)$ and $(ii)$ are equivalent. We leave the details to the reader.
\end{proof}

\subsection{Universal Jamison spaces for $C_0$-semigroups}
We can prove an analog of Theorem \ref{theoremejamisonuniversel} in the context of $C_0$-semigroups. To begin with, we define the notion of \textit{universal Jamison space for} $C_0$-\textit{semigroups}.

\begin{Def}\label{defuniversaljamisonsg}
Let $X$ be a separable infinite-dimensional complex Banach space. We say that $X$ is a \textit{universal Jamison space for} $C_0$-\textit{semigroups} if the following property holds true: for any increasing sequence of positive real numbers $(t_k)_{k\ge 0}$ which is not a Jamison sequence for $C_0$-semigroups, there exists a $C_0$-semigroup $(T_t)_{t\ge 0}$ of bounded linear operators on $X$ (with infinitesimal generator $A$) which is bounded with respect to the sequence $(t_k)_{k\ge 0}$ and such that the set $\sigma_p(A)\cap i\R$ is uncountable.
\end{Def}

The analog of Theorem \ref{theoremejamisonuniversel} in the context of $C_0$-semigroup is the following.

\begin{Theo}\label{espacejamisonuniverselsemigroupes}
Let $X$ be a separable complex Banach space which admits an unconditional Schauder decomposition. Then $X$ is a universal Jamison space for $C_0$-semigroups.
\end{Theo}

\begin{proof}
Let $(t_k)_{k\ge 0}$ be an increasing sequence of positive real numbers such that $t_0=1$ and $t_k \longrightarrow +\infty$. We assume that $(t_k)_{k\ge 0}$ is not a Jamison sequence for $C_0$-semigroups and we need to construct a $C_0$-semigroup (with infinitesimal generator $A$) on the space $X$ which is bounded with respect to the sequence $(t_k)_{k\ge 0}$ and such that the set $\sigma_p(A)\cap i\R$ is uncountable. For any nonnegative integer $k$, we denote by $n_k$ the integer part of $t_k$ (in particular $n_0=1$). Since $(t_k)_{k\ge 0}$ is not a Jamison sequence for $C_0$-semigroups, we know from Theorem \ref{equivalencejamison} that the sequence $(n_k)_{k\ge 0}$ is not a Jamison sequence. We then can consider the operator $T=D+B$ which has been constructed in the proof of Theorem \ref{theoremejamisonuniversel}: it is partially power-bounded with respect to the sequence $(n_k)_{k\ge 0}$ and the set $\sigma_p(T)\cap \T$ is uncountable. In the proof of Theorem \ref{theoremejamisonuniversel}, we can choose $\lambda_1$ very close to $1$. For instance, we can take this unimodular complex number in such a way that $\vert\lambda_1-1\vert=\frac{1}{3}$. Then we choose $\lambda_n$ on the arc between $1$ and $\lambda_1$  for any positive integer $n$. We have $\sigma(D)=\overline{\big\{\lambda_n\,;\, n\ge 1\big\}}$ and if we take $\lambda_n$ sufficiently close to $\lambda_{j(n)}$, we get $\vert\vert B\vert\vert <\frac{1}{3}$ which implies that $\sigma(T)\subset \sigma(D)_{1/3}$. Here we denote by $K_\varepsilon$ the set
\begin{equation*}
K_\varepsilon :=\big\{z\in \C\,;\, \mathrm{dist}(z,K)<\varepsilon\big\},
\end{equation*}
where $K$ is a compact set of the complex plane. In particular, the set $\sigma(T)$ is contained in $\mathcal{P}_{1/2}:=\big\{z\in \C\,;\, \mathfrak{Re}\, z>\frac{1}{2}\big\}$. Since the complex logarithm is an analytic function in this domain, the operator $\mathrm{Log}\,T$ is well-defined (by the functional calculus) and bounded on $X$. Then the $C_0$-semigroup $(T_t)_{t\ge 0}$ with infinitesimal generator $\mathrm{Log}\,T$ is such that for any $t\ge 0$, $ T_t=e^{t\,\mathrm{Log}\, T}.$ Since for any $z\in \mathcal{P}_{1/2}$, the equality $e^{n\,\mathrm{Log}\,z}=z^n$ holds true, we conclude that for any nonnegative integer $k$, we have $T_{n_k}=T^{n_k}$. Furthermore, we know from Fact \ref{faitjamison} that the semigroup $(T_t)_{t\ge 0}$ is bounded with respect to the sequence $(t_k)_{k\ge 0}$ and since
\begin{equation*}
\sigma_p(T)\cap \mathbb{T}=\sigma_p(T_1)\cap\mathbb{T}=\exp(\sigma_p(\mathrm{Log}\, T)\cap i\mathbb{R}),
\end{equation*}
this yields that the set $\sigma_p(\mathrm{Log}\, T)\cap i\mathbb{R}$ is uncountable.
\end{proof}

We finish this paper by proving a result concerning the Hausdorff dimension of $\sigma_p(A)\cap i\R$ which fits into the framework of \cite{BadeaGrivaux2}.

\subsection{Hausdorff dimension of $\sigma_p(A)\cap i\R$}
Ransford and Roginskaya proved in \cite{RansfordRoginskaya} that if a $C_0$-semigroup (with infinitesimal generator $A$) of bounded linear operators on a separable complex Banach space is bounded with respect to an increasing sequence of positive real numbers which tends to infinity, then the set $\sigma_p(A)\cap i\R$ has Lebesgue measure zero (\cite[Theorem 4.1]{RansfordRoginskaya}). It is then natural to study the Hausdorff dimension of $\sigma_p(A)\cap i\R$. The same authors proved that the Hausdorff dimension $\dim_H(\sigma_p(A)\cap i\R)$ of $\sigma_p(A)\cap i\R$ can be controlled by the growth of the sequence $(t_k)_{k\ge 0}$ (\cite[Theorem 4.1]{RansfordRoginskaya}). The results below are the same as in \cite[Section 3]{BadeaGrivaux2} but in the context of $C_0$-semigroups. To begin with, we prove the following theorem.

\begin{Theo}\label{hausdorff}
Let $(n_k)_{k\ge 0}$ be an increasing sequence of positive integers such that $n_0=1$. Let $\mathcal{S}$ be any class of subsets of the unit circle $\T$ such that every subset of an element of $\mathcal{S}$ is an element of $\mathcal{S}$ itself. The following assertions are equivalent$:$\\
\noindent $(1)$ for every separable complex Banach space $X$ and every $C_0$-semigroup $(T_t)_{t\ge 0}$ of bounded linear operators on $X$ which is bounded with respect to the sequence $(n_k)_{k\ge 0}$, the set $\sigma_p(T_1)\cap \T$ belongs to the set $\mathcal{S};$\\
\noindent $(2)$ for every subset $K$ of $\T$ not belonging to $\mathcal{S}$, the metric space $(K,d_{(n_k)})$ is non-separable$;$\\
\noindent $(3)$ for every subset $K$ of $\T$ not belonging to $\mathcal{S}$, there exists $\varepsilon>0$ such that $K$ contains an uncountable $\varepsilon$-separated family for the distance $d_{(n_k)}$.
\end{Theo}

\begin{proof}
It is clear that assertions $(2)$ and $(3)$ are equivalent. We now prove $(3) \Rightarrow (1)$. Let $(T_t)_{t\ge 0}$ be a $C_0$-semigroup of bounded linear operators on $X$ such that $M:=\sup_{k\ge 0}\vert\vert T_{n_k}\vert\vert<+\infty$ and assume that  $\sigma_p(T_1)\cap \mathbb{T}\notin \mathcal{S}$. According to $(3)$, there exists $\varepsilon>0$ such that $\sigma_p(T_1)\cap \T$ contains an uncountable $\varepsilon$-separated family for the distance $d_{(n_k)}$. Let $\lambda$ and $\mu$ be unimodular eigenvalues of $T_1$ with associated eigenvectors (of norm $1$) $e_\lambda$ and $e_\mu$. As in the proof of Theorem \ref{jamisoncontinu}, we obtain
\begin{equation*}
\vert\vert e_\lambda- e_\mu\vert\vert\ge \frac{d_{(n_k)}(\lambda,\mu)}{M+1}\ge \frac{\varepsilon}{M+1},
\end{equation*}
which contradicts the separability of $X$. The proof of $(1) \Rightarrow (2)$ is the same as that of Theorem \ref{jamisoncontinu}.
\end{proof}

Let $\mathcal{S}$ stand for the Borel subsets of $\T$ with Hausdorff dimension strictly less than $1$. We get then the following corollary.

\begin{Cor}\label{hausdorffcorollaire}
Let $(n_k)_{k\ge 0}$ be an increasing sequence of positive integers such that $n_0=1$. The following assertions are equivalent$:$\\
\noindent $(1)$ there exists a separable complex Banach space $X$ and a $C_0$-semigroup $(T_t)_{t\ge 0}$ of bounded linear operators on $X$ which is bounded with respect to the sequence $(n_k)_{k\ge 0}$ and such that the set $\sigma_p(T_1)\cap \T$ is of Hausdorff dimension equal to $1;$\\
\noindent $(2)$ there exists a subset $K$ of $\T$ of Hausdorff dimension equal to $1$ such that the metric space $(K,d_{(n_k)})$ is separable.
\end{Cor}

When the sequence $(n_k)_{k\ge 0}$ is such that $\frac{n_{k+1}}{n_k} \longrightarrow +\infty$, C. Badea and S. Grivaux showed that there exists a subset $K$ of $\T$ of Hausdorff dimension equal to $1$ such that the metric space $(K,d_{(n_k)})$ is separable (see \cite[Theorem 3.4]{BadeaGrivaux2}). Then we have the following result.

\begin{Theo}\label{hausdorff1}
Let $(t_k)_{k\ge 0}$ be an increasing sequence of positive real numbers such that $t_0=1$ and $\frac{t_{k+1}}{t_k} \longrightarrow +\infty$. Then there exists a separable complex Banach space $X$ and a $C_0$-semigroup $(T_t)_{t\ge 0}$ $($with infinitesimal generator $A)$ of bounded linear operators on $X$ which is bounded with respect to the sequence $(t_k)_{k\ge 0}$ and such that $\dim_H(\sigma_p(A)\cap i\R)=1$.
\end{Theo}

\begin{proof}
For any nonnegative integer $k$, we denote by $n_k$ the integer part of $t_k$. In particular, $n_0=1$. Since $\frac{t_{k+1}}{t_k} \longrightarrow +\infty$, we have $\frac{n_{k+1}}{n_k} \longrightarrow +\infty$ as well. But we know from the proof of \cite[Theorem 3.4]{BadeaGrivaux2} that there exists a subset $K$ of $\T$ of Hausdorff dimension equal to $1$ such that the metric space $(K,d_{(n_k)})$ is separable. According to Corollary \ref{hausdorffcorollaire}, there exists a separable complex Banach space $X$ and a $C_0$-semigroup $(T_t)_{t\ge 0}$ (with infinitesimal generator $A$) of bounded linear operators on $X$ which is bounded with respect to the sequence $(n_k)_{k\ge 0}$ and such that  $\dim_H(\sigma_p(T_1)\cap \mathbb{T})=1$. According to Fact \ref{faitjamison}, the semigroup $(T_t)_{t\ge 0}$ is also bounded with respect to the sequence $(t_k)_{k\ge 0}$. Furthermore, it is well-known that a Lipschitz function decreases the Hausdorff dimension. If  $f: i\R \longrightarrow \C$ is the exponential function, then $f$ is a Lispchitz function and we have the equality $\sigma_p(T_1)\cap \mathbb{T}=f(\sigma_p(A)\cap i\R)$. We deduce from this that
\begin{equation*}
1=\dim_H(\sigma_p(T_1)\cap \mathbb{T})\le \dim_H(\sigma_p(A)\cap i\R)\le 1.
\end{equation*}
We thus conclude that the Hausdorff dimension of the set $\sigma_p(A)\cap i\R$ is equal to $1$.
\end{proof}

\noindent\emph{Acknowledgement}: I am grateful to the referee for valuable suggestions on the presentation of the paper. I am also grateful to my advisors, Catalin Badea and Sophie Grivaux, for helpful discussions on Jamison sequences.

\noindent \textit{Vincent Devinck}\newline
\noindent \textit{Laboratoire Paul Painlev\'e}\newline
\noindent \textit{UMR 8524}\newline
\noindent \textit{Universit\'e des Sciences et Technologies de Lille}\newline
\noindent \textit{Cit\'e Scientifique}\newline
\noindent \textit{59655 Villeneuve d'Ascq cedex}\newline
\noindent \textit{France}\newline
\noindent \texttt{devinck.vincent@gmail.com}

\end{document}